\documentclass[12pt]{amsart}
\usepackage{amssymb,amsmath, amsthm,latexsym,verbatim}
\usepackage{a4wide}
\usepackage[hypertex]{hyperref}
\usepackage{enumitem}

\theoremstyle{plain}

\newtheorem{lem}{Lemma}[section]
\newtheorem{theo}[lem]{Theorem}
\newtheorem{prop}[lem]{Proposition}
\newtheorem{corollary}[lem]{Corollary}

\parskip=\medskipamount
\font\k=cmr7
\font\rm=cmr12

  \newcommand {\reg}{\mbox{\k reg}}

  \newcommand{\s}{\operatorname{s}}
  \newcommand {\C}{{\mathbb C}}
  \newcommand {\bH}{{\mathbb H}}
  \newcommand {\N}{{\mathbb N}}
  \newcommand {\R}{{\mathbb R}}

  \newcommand {\af}{{\mathfrak a}}
  \newcommand {\gf}{{\mathfrak g}}
  \newcommand {\mf}{{\mathfrak m}}
  \newcommand {\kf}{{\mathfrak k}}
  
  \newcommand {\pf}{{\mathfrak p}}

  \newcommand {\nf}{{\mathfrak n}}

  \newcommand {\Co}{{\mathcal C}}

\newcommand{\cE}{\mathcal{E}}
\newcommand{\CC}{\operatorname{C}}
\newcommand  {\cZ}{{\mathcal Z}}
\newcommand {\bs}{\backslash}

 \newcommand {\bx}{{\bf x}}
 \newcommand {\by}{{\bf y}}

\renewcommand{\Re}{\operatorname{Re}}
\newcommand{\Tr}{\operatorname{Tr}}

\newcommand{\End}{\operatorname{End}}

\newcommand{\tr}{\operatorname{tr}}
\newcommand{\Id}{\operatorname{Id}}
\newcommand{\Hom}{\operatorname{Hom}}

\newcommand{\vol}{\operatorname{vol}}

\newcommand{\SL}{\operatorname{SL}}
\newcommand{\GL}{\operatorname{GL}}
\newcommand{\SO}{\operatorname{SO}}

\newcommand{\Spin}{\operatorname{Spin}}
\newcommand{\Rep}{\operatorname{Rep}}
\newcommand{\Ad}{\operatorname{Ad}}

\newcommand{\supp}{\operatorname{supp}}

\newcommand{\arcosh}{\operatorname{arcosh}}

\renewcommand{\det}{\operatorname{det}}

\newcommand{\pro}{\operatorname{pr}}

\newcommand{\ve}{\varepsilon}

\begin{document}

\title[]
{On the analytic torsion of hyperbolic manifolds of finite 
volume}
\date{\today}

\author{Werner M\"uller}
\address{Universit\"at Bonn\\
Mathematisches Institut\\
Endenicher Allee 60\\
D -- 53115 Bonn, Germany}
\email{mueller@math.uni-bonn.de}


\begin{abstract}
In this paper we study  the analytic torsion for a complete
oriented hyperbolic manifold of finite volume. This requires the definition
of a regularized trace of heat operators. We use the Selberg trace formula
to study the asymptotic behavior of the regularized trace for small time.
The main result of the paper is a new approach to deal with the weighted 
orbital integrals on the geometric side of the trace formula.
\end{abstract}

\maketitle

\setcounter{tocdepth}{1}
\section{Introduction}
\setcounter{equation}{0}

Let $G=\SO_0(d,1)$ and $K=\SO(d)$. Then $K$ is a maximal compact subgroup of
$G$. Let $\widetilde X:=G/K$. Equipped with a suitably normalized  invariant 
metric, $\widetilde X$ is isometric to the hyperbolic space $\bH^d$ of 
dimension $d$. Let $\Gamma\subset G$ be a lattice, i.e. a discrete
subgroup with $\vol(\Gamma\bs G)<\infty$.  Assume that $\Gamma$ is torsion free.
Then $X:=\Gamma\bs \widetilde X$ is an oriented hyperbolic $d$-manifold
of finite volume. Let $\tau$ be an irreducible finite dimensional complex
representation of $G$. Let $E_\tau$ be the flat vector bundle associated to the
restriction of $\tau$ to $\Gamma$. By \cite{MM}, $E_\tau$ can be equipped with
a canonical Hermitian fibre metric, called admissible, which is unique up to
scaling. Let $\Delta_p(\tau)$ be the Laplace operator acting in the space of
 $E_\tau$-valued $p$-forms with respect to the metrics on $X$ and in $E_\tau$. 
In \cite{MP2}
we introduced the analytic torsion $T_X(\tau)$. If $X$ is compact, $T_X(\tau)$
is defined in the usual way \cite{RS} by
\begin{equation}\label{anator1}
\log T_X(\tau)=\frac{1}{2}\sum_{p=1}^d(-1)^p p \frac{d}{ds}\left(
\frac{1}{\Gamma(s)}\int_0^\infty\left(\Tr(e^{-t\Delta_p(\tau)})-b_p(\tau)\right)
t^{s-1}\;dt\right)\bigg|_{s=0},
\end{equation}
where $b_p(\tau)=\dim\ker(\Delta_p(\tau))$ and the right hand side is defined
near $s=0$ by analytic continuation.
In the non-compact case the Laplace operator $\Delta_p(\tau)$ has a nonempty 
continuous spectrum and hence, $e^{-t\Delta_p(\tau)}$ is not a trace class operator.
In \cite{MP2} we introduced the regularized trace 
$\Tr_{\reg}\left(e^{-t\Delta_p(\tau)}\right)$ of the heat operator which we used to
define $T_X(\tau)$ by the analogous formula \eqref{anator1} with the usual 
trace replaced by the regularized trace. In order to show that the Mellin 
transform of the regularized trace is defined in some half plane and admits
a meromorphic extension to the whole complex plane one needs to know the
behavior of $\Tr_{\reg}\left(e^{-t\Delta_p(\tau)}\right)$ as $t\to 0$ and 
$t\to\infty$. To establish an asymptotic expansion of 
$\Tr_{\reg}\left(e^{-t\Delta_p(\tau)}\right)$ as $t\to 0$ we used the Selberg
trace formula. The difficult part is to deal with the weighted orbital
integrals occurring on the geometric side of the trace formula.
In fact, in \cite{MP2} we use the invariant trace 
formula of \cite{Ho1}. 
In this way the parabolic distributions become 
invariant distributions. To deal with these distributions we applied the 
Fourier 
inversion formula established by W. Hoffmann \cite{Ho2}. This 
is a very 
heavy and quite complicated machinery. Moreover at present, except for 
$\SL(3,\R)$, it is not available in  higher rank. 
The main purpose of this paper is to develop a more simplified method to deal
with the parabolic contribution to the trace formula, which also has a chance 
to be extended to the higher rank case.

Next we explain some details of our method. For simplicity we assume that $d$ 
is odd.  We also assume that $\Gamma$ satisfies \eqref{neat}.
Let $\nu\colon K\to\GL(V)$ be an irreducible unitary representation of $K$. Let
$\widetilde E_\nu\to X$ be the associated homogeneous vector bundle over 
$\widetilde X$ and let $E_\nu=\Gamma\bs\widetilde E_\nu$ be the corresponding 
locally homogeneous vector bundle over $X$. Let $\nabla^\nu$ be the invariant
connection in $E_\nu$ and let $\Delta_\nu=(\nabla^\nu)^*\nabla^\nu$ be the 
associated Bochner-Laplace operator. Let $A_\nu$ denote the differential
operator which is induced in $C^\infty(X,E_\nu)$ by the action of $-\Omega$, where
$\Omega\in\cZ(\gf_\C)$ is the Casimir element. Then $\Delta_\nu=A_\nu+
\nu(\Omega_K)$, where $\Omega_K\in\cZ(\kf_\C)$ is the Casimir element of $K$.
Note that $\nu(\Omega_K)$ is a scalar. Hence $A_\nu$ is essentially self-adjoint
and bounded from below. Therefore the heat semigroup $e^{-tA_\nu}$ is well
defined.  The study of the regularized trace of 
$e^{-t\Delta_p(\tau)}$ can be reduced to the study of the regularized 
trace of the heat operators $e^{-tA_\nu}$. Let $H^\nu(t,x,y)$ be the kernel of 
$e^{-tA_\nu}$. For $Y>1$ sufficiently large let $X(Y)$
be the compact manifold with boundary obtained from $X$ by truncating $X$ at
level $Y$. It was shown in \cite{MP2} that there exists $\alpha(t)\in \R$
such that $\int_{X(Y)}\tr H^\nu(t,x,x)\,dx -\alpha(t)\log Y$ has a limit as
$Y\to\infty$. Then we put
\[
\Tr\left(e^{-tA_\nu}\right):=\lim_{Y\to\infty}\left(\int_{X(Y)}
\tr H^\nu(t,x,x)\;dx - \alpha(t)\log Y\right).
\]
Our main result is the following theorem.
\begin{theo}\label{theo-main}
For every $\nu\in\widehat K$ there is an asymptotic expansion 
\begin{equation}
\Tr_{\reg}\left(e^{-tA_\nu}\right)\sim\sum_{j=0}^\infty a_j(\nu)t^{-d/2+j/2}\log t
+\sum_{k=0}^\infty b_k(\nu)t^{-d/2+j/2}
\end{equation}
as $t\to 0$. Moreover $a_n(\nu)=0$.
\end{theo}
To prove this theorem, we use the Selberg trace formula as in \cite{MP2}. 
The connection with the Selberg trace formula is as follows. Let $\widetilde
A^\nu$ be the lift of $A_\nu$  to the universal
covering $\widetilde X$ of $X$. The heat operator $e^{-t\widetilde A_\nu}$ is a 
convolution operator with a smooth kernel $H^\nu_t\colon G\to \End(V)$. Let
$h^\nu_t\in C^\infty(G)$ be defined by
\[
h^\nu_t:=\tr H^\nu_t(g),\quad g\in G.
\]
In fact, $h^\nu_t$ belongs to Harish-Chandra's Schwartz space $\Co^1(G)$.
In \cite{MP2} it was proved that $\Tr_{\reg}\left(e^{-tA_\nu}\right)$ equals
the spectral side of the Selberg trace formula with respect to the test 
function $h^\nu_t$, where the spectral side means the sum of all terms
 corresponding to the discrete and continuous spectrum in the trace formula. 
Applying the trace formula we are led to the following equality
\[
\Tr_{\reg}\left(e^{-tA_\nu}\right)=I(h^\nu_t)+H(h^\nu_t)+T(h^\nu_t)+T^\prime(h^\nu_t),
\]
where $I$, $H$, $T$ and $T^\prime$ are certain distributions which are
associated to the identity, the hyperbolic conjugacy classes and the 
parabolic conjugacy classes of $\Gamma$, respectively (see \cite{Wa1}).
To study the asymptotic behavior of $I(h_t^{\tau,p})$ as $t\to 0$, one can use
the Plancherel theorem. Since $G$ has $\R$-rank one, one can use the Fourier
inversion formula for regular semi-simple orbital integrals to study 
$H(h_t^{\tau,p})$. It follows that $H(h_t^{\tau,p})$ is exponentially decreasing 
as $t\to 0$. The distribution $T$ is invariant and can be expressed in 
terms of characters. This leads to an asymptotic expansion of $T(h^\nu_t)$ as
$t\to 0$. What remains is to deal with the distribution $T^\prime$, which is not
invariant. Using standard estimates of heat kernels, it can be reduced to the
study of integrals of the form
\begin{equation}\label{integral}
\int_{\R^{d-1}}e^{-f(x)/t}g(x)\log\| x\|\;dx,
\end{equation} 
where $g\in C^\infty_c(\R^{d-1})$ and $f\in C^\infty(\R^{d-1})$ has an isolated
critical point at $x=0$ of index zero. Then we use the method of the stationary 
phase approximation to determine the asymptotic behavior of this integral as 
$t\to 0$. 

The paper is organized as follows. In section \ref{sec-prel} we fix notations
and collect some basic facts. In section \ref{sec-regtrace} we define the
regularized trace of the heat operators $e^{-tA_\nu}$ and relate it to
the spectral side of the Selberg trace formula where the test function 
is obtained from the kernel of the
heat operator on the universal covering. In section \ref{sec-trform}
we apply the Selberg trace formula to express the regularized trace through
the geometric side of the trace formula. Then we determine the asymptotic
behavior of all terms except the weighted orbital integral. Section
\ref{sec-heatker} is a preparatory section where we establish estimates of 
heat kernels and describe their asymptotic expansion for small time. This is 
used in section
\ref{sec-weightorb} to study the asymptotic behavior of the weighted
orbital integrals. Applying the results of the previous section, the problem is
reduced to the study of integrals of the form \eqref{integral}. To deal
with these integrals we apply the method of the stationary phase
approximation. This leads finally to the proof of Theorem \ref{theo-main}. In
the last section \ref{sec-analytor} we discuss the analytic torsion.

\bigskip
{\bf Acknowledgment.}
I would like to thank Andras Vasy for some very useful hints.

\section{Preliminaries}\label{sec-prel}
\setcounter{equation}{0}

\subsection{}
Let $d=2n+1$, $n\in\N$. Let either $G=\SO_0(d,1)$, $K=\SO(d)$ or
$G=\Spin(d,1)$, $K=\Spin(d)$. Then $K$ 
is a maximal compact subgroup of $G$. Put $\widetilde X=G/K$. Let
\[
G=NAK
\]
be the standard Iwasawa decomposition of $G$ and let $M$ be the 
centralizer of $A$ in $K$. Then $M=\SO(d-1)$ or $M=\Spin(d-1)$. 
Let $\gf$, $\nf$, $\af$, $\kf$, $\mf$ denote the Lie algebras
of $G$, $N$, $A$, $K$ and $M$, respectively. Define the 
standard Cartan involution $\theta:\gf\rightarrow \gf$ by
\begin{align*}
\theta(Y)=-Y^{t},\quad Y\in\gf. 
\end{align*}
The lift of $\theta$ to $G$ will be denoted by the same letter $\theta$. Let 
\begin{equation}\label{cartan}
\gf=\kf\oplus\pf
\end{equation}
be the Cartan decomposition of $\gf$ with respect to $\theta$. Let $x_0=eK\in
\tilde X$. Then we have a canonical isomorphism
\begin{equation}\label{tspace}
T_{x_0}\tilde X\cong \pf.
\end{equation}
Define the symmetric bi-linear form $\langle\cdot,\cdot\rangle$ on $\gf$ by
\begin{equation}\label{killnorm}
\langle Y_1,Y_2\rangle:=\frac{1}{2(d-1)}B(Y_1,Y_2),\quad Y_1,Y_2\in\gf.
\end{equation}
By \eqref{tspace} the restriction of $\langle\cdot,\cdot\rangle$ to $\pf$ 
defines an inner product on $T_{x_0}\tilde X$ and therefore an invariant 
metric on $\tilde X$. This metric has constant curvature $-1$. Then $\tilde X$,
equipped with this metric, is isometric to the hyperbolic space
${\mathbb{H}}^{d}$. 

Fix a Cartan subalgebra $\mathfrak{b}$ of $\mathfrak{m}$. 
Then
\begin{align*}
\mathfrak{h}:=\mathfrak{a}\oplus\mathfrak{b}
\end{align*}
is a Cartan subalgebra of $\mathfrak{g}$. We can identify
$\mathfrak{g}_\C\cong\mathfrak{so}(d+1,\C)$. Let $e_1\in\af^*$ be the
positive restricted root defining $\mathfrak{n}$.
Then we fix $e_2,\dots,e_{n+1}\in
i\mathfrak{b}^*$ such that 
the positive roots $\Delta^+(\mathfrak{g}_\C,\mathfrak{h}_\C)$ are chosen as in
\cite[page 684-685]{Kn2}
for the root system $D_{n+1}$. We let
$\Delta^+(\mathfrak{g}_\C,\mathfrak{a}_\C)$ be
the set of roots of $\Delta^+(\mathfrak{g}_\C,\mathfrak{h}_\C)$ which do not
vanish on $\af_\C$. The positive roots
$\Delta^+(\mathfrak{m}_\C,\mathfrak{b}_\C)$
are chosen such that they are restrictions of elements from
$\Delta^+(\mathfrak{g}_\C,\mathfrak{h}_\C)$.
For $j=1,\dots,n+1$ let
\begin{equation}\label{rho}
\rho_{j}:=n+1-j.
\end{equation}
Then the half-sums of positive roots $\rho_G$ and $\rho_M$, respectively, are
given by
\begin{align}\label{Definition von rho(G)}
\rho_{G}:=\frac{1}{2}\sum_{\alpha\in\Delta^{+}(\mathfrak{g}_{\mathbb{C}},
\mathfrak{h}_\mathbb{C})}\alpha=\sum_{j=1}^{n+1}\rho_{j}e_{j};\quad
\rho_{M}:=\frac{1}{2}\sum_{\alpha\in\Delta^{+}(\mathfrak{m}_{\mathbb{C}},
\mathfrak{b}_{\mathbb{C}})}\alpha=\sum_{j=2}^{n+1}\rho_{j}e_{j}.
\end{align}
Let $\Omega$, $\Omega_K$ and $\Omega_M$ be the Casimir elements of $G$, $K$
and $M$, respectively, with respect to the normalized Killing form
\eqref{killnorm}. 

\subsection{}
Let ${{\mathbb{Z}}\left[\frac{1}{2}\right]}^{j}$ be the set of all 
$(k_{1},\dots,k_{j})\in\mathbb{Q}^{j}$ such that either all $k_{i}$ are 
integers or all $k_{i}$ are half integers. 
Let $\Rep(G)$ denote the set of finite dimensional irreducible
representations 
$\tau$ of $G$. These  
are parametrized by their highest weights
\begin{equation}\label{Darstellungen von G}
\Lambda(\tau)=k_{1}(\tau)e_{1}+\dots+k_{n+1}(\tau)e_{n+1};\:\:
k_{1}(\tau)\geq 
k_{2}(\tau)\geq\dots\geq k_{n}(\tau)\geq \left|k_{n+1}(\tau)\right|,
\end{equation}
where $(k_{1}(\tau),\dots, k_{n+1}(\tau))$ belongs to
${{\mathbb{Z}}\left[\frac{1}{2}\right]}^{n+1}$ if 
$G=\Spin(d,1)$ and to ${{\mathbb{Z}}}^{n+1}$ if $G=\SO^0(d,1)$. 
Moreover, the finite dimensional irreducible representations $\nu\in\hat{K}$ of
$K$ are
parametrized by their highest weights
\begin{equation}\label{Darstellungen von K}
\Lambda(\nu)=k_{2}(\nu)e_{2}+\dots+k_{n+1}(\nu)e_{n+1};\:\:
k_{2}(\nu)\geq 
k_{3}(\nu)\geq\dots\geq k_{n}(\nu)\geq k_{n+1}(\nu)\geq 0,
\end{equation} 
where $(k_{2}(\nu),\dots, k_{n+1}(\nu))$ belongs to
${{\mathbb{Z}}\left[\frac{1}{2}\right]}^{n}$ if 
$G=\Spin(d,1)$ and to ${{\mathbb{Z}}}^{n}$ if $G=\SO^0(d,1)$.
Finally, the  finite dimensional irreducible representations 
$\sigma\in\hat{M}$ of $M$ 
are parametrized by their highest weights
\begin{equation}\label{Darstellungen von M}
\Lambda(\sigma)=k_{2}(\sigma)e_{2}+\dots+k_{n+1}(\sigma)e_{n+1};\:\:
k_{2}(\sigma)\geq 
k_{3}(\sigma)\geq\dots\geq k_{n}(\sigma)\geq \left|k_{n+1}(\sigma)\right|,
\end{equation}
where $(k_{2}(\sigma),\dots, k_{n+1}(\sigma))$ belongs to
${{\mathbb{Z}}\left[\frac{1}{2}\right]}^{n}$, if 
$G=\Spin(d,1)$, and to ${{\mathbb{Z}}}^{n}$, if $G=\SO^0(d,1)$.
For $\nu\in\hat{K}$ and  $\sigma\in \hat{M}$ we denote by $[\nu:\sigma]$ the
multiplicity of $\sigma$ in the restriction of $\nu$ to $M$.

Let $M'$ be the normalizer of $A$ in $K$ and let $W(A)=M'/M$ be the 
restricted Weyl-group. It has order two and it acts on the finite-dimensional 
representations of $M$ as follows. Let $w_0\in W(A)$ be the non-trivial 
element and let $m_0\in M^\prime$ be a representative of $w_0$. Given 
$\sigma\in\hat M$, the representation $w_0\sigma\in \hat M$ is defined by
\[
w_0\sigma(m)=\sigma(m_0mm_0^{-1}),\quad m\in M.
\] 
Let
$\Lambda(\sigma)=k_{2}(\sigma)e_{2}+\dots+k_{n+1}(\sigma)e_{n+1}$ be the 
highest weight
of $\sigma$ as in \eqref{Darstellungen von M}. Then the highest weight 
$\Lambda(w_0\sigma)$ of $w_0\sigma$ is given by
\begin{equation}\label{wsigma}
\Lambda(w_0\sigma)=k_{2}(\sigma)e_{2}+\dots+k_{n}(\sigma)e_{n}
-k_{n+1}(\sigma)e_{n+1}.
\end{equation}

\subsection{}

Let $P:=NAM$. This is the standard parabolic subgroup of $G$. 
We equip $\af$ with the norm induced from the restriction of the
normalized Killing form on $\gf$. Let $H_1\in\af$ be
the unique vector which is of norm one and such that the positive restricted
root,
implicit in the choice of $N$, is positive on $H_1$. Let
$\exp:\af\to
A$ be
the exponential map.
Every $a\in A$ can be written 
as $a=\exp{\log{a}}$, where $\log{a}\in\mathfrak{a}$ is unique.
For $t\in\mathbb{R}$, we let 
$a(t):=\exp{(tH_{1})}$. If $g\in G$, we define $n(g)\in N$, $
H(g)\in \R$ and $\kappa(g)\in K$ by 
\begin{align*}
g=n(g)a(H(g))\kappa(g). 
\end{align*}
A given $g\in G$ can always be written in the form
\begin{equation}\label{radcomp}
g=k_1a(t(g))k_2,
\end{equation}
where $k_1,k_2\in K$ and $t(g)\ge 0$. We note that $t(g)$ is unique and we
call it the radial component of $g$. For $x,y\in \tilde X$ let $r(x,y)$
denote the geodesic distance of $x$ and $y$. Then we have
\begin{equation}\label{dist}
r(g(x_0),x_0)=t(g),\quad g\in G.
\end{equation}

Now let $P'$ be any proper parabolic subgroup of $G$. Then 
there exists a $k_{P'}\in K$ such that
$P'=N_{P'}A_{P'}M_{P'}$ with $N_{P'}=k_{P'}Nk_{P'}^{-1}$,
$A_{P'}=k_{P'}Ak_{P'}^{-1}$, $M_{P'}=k_{P'}Mk_{P'}^{-1}$. We choose a set of
$k_{P'}$'s, which will be fixed from now on. Let
$k_{P}=1$.
We let $a_{P'}(t):=k_{P'}a(t)k_{P'}^{-1}$. If $g\in G$, we define
$n_{P'}(g)\in N_{P'}$, $H_{P'}(g)\in
\mathbb{R}$ and $\kappa_{P'}(g)\in K$ by 
\begin{align}\label{eqg}
g=n_{P'}(g)a_{P'}(H_{P'}(g))\kappa_{P'}(g) 
\end{align}
and we define an identification $\iota_{P'}$ of
$\left(0,\infty\right)$ with
$A_{P'}$ by $\iota_{P'}(t):=a_{P'}(\log(t))$. 
For $Y>0$, let $A^{0}_{P'}\left[Y\right]:=\iota_{P'}(Y,\infty)$ and
$A_{P'}\left[Y\right]:=\iota_{P'}[Y,\infty)$. For $g\in G$ as in \eqref{eqg} we
let $y_{P'}(g):=e^{H_{P'}(g)}$.

\subsection{}
Let $\Gamma$ be a discrete subgroup of $G$ such that
$\vol(\Gamma\backslash G)<\infty$.  Let 
$X:=\Gamma\backslash\widetilde{X}$. Let $\pro_X:G\to X$ be the projection. 
A parabolic subgroup $P'$ of $G$ is called a $\Gamma$-cuspidal parabolic
subgroup if $\Gamma\cap N_{P'}$ is a lattice in $N_{P'}$. We assume that that 
$\Gamma$ satisfies the following condition: For every $\Gamma$-cuspidal proper
parabolic subgroup $P=N_PA_PM_P$ of $G$ we have
\begin{equation}\label{neat}
\Gamma\cap P=\Gamma\cap N_P.
\end{equation}
We note that this condition is satisfied, if $\Gamma$ is ``neat'', which means
that the group generated by the eigenvalues of any $\gamma\in\Gamma$ contains 
no roots of unity $\neq 1$. 

Let $\mathfrak{P}_\Gamma=\{P_1,\dots,P_{\kappa(\Gamma)}\}$ be a set of 
representatives of $\Gamma$-conjugacy classes of $\Gamma$-cuspidal parabolic 
subgroups of $G$. 
The number
\begin{equation}\label{cusps}
\kappa(X):=\kappa(\Gamma)=\#\mathfrak{P}_{\Gamma}
\end{equation}
is finite and equals the number of cusps of $X$. More precisely, for each
$P_i\in\mathfrak{P}_\Gamma$ there exists a $Y_{P_i}>0$ and there exists 
a compact connected subset $C=C(Y_{P_1},\dots,Y_{P_{\kappa(\Gamma)}})$
of $G$  such that in the sense of a disjoint union one has
\begin{align}\label{FBI}
G=\Gamma\cdot
C\sqcup\bigsqcup_{i=1}^{
\kappa(X)}\Gamma\cdot
N_{P_i}A^{0}_{P_i}\left[Y_{P_i}\right]K
\end{align}
and such that 
\begin{align}\label{FBII}
\gamma\cdot N_{P_i}A^0_{P_i}\left[Y_{P_i}\right]K\cap
N_{P_i}A_{P_i}^{0}\left[Y_{P_i}\right]K\neq
\emptyset\Leftrightarrow\gamma\in \Gamma\cap P_i. 
\end{align}
We define the height-function $y_{\Gamma,P_i}$ on $X$ by 
\begin{equation}\label{heightf}
y_{\Gamma,P_i}(x):=\sup\{y_{P_i}(g)\colon g\in G,\; \pro_X(g)=x\}. 
\end{equation}
By \eqref{FBI} and \eqref{FBII} the supremum is finite.
For $Y\in \R^+$ let
\begin{equation}\label{truncmfd}
X(Y):=\{x\in X\colon
y_{\Gamma, P_i}(x)\leq Y,\:i=1,\dots,\kappa(X)\}.
\end{equation}

\subsection{}\label{secPL}
Recall that $d=2n+1$. For $\sigma\in\hat{M}$ and $\lambda\in\R$ let
$\mu_{\sigma}(\lambda)$ 
be the Plancherel measure associated to the principal series 
representation $\pi_{\sigma,\lambda}$. Then, since 
${\rm{rk}}(G)>{\rm{rk}}(K)$, $\mu_{\sigma}(\lambda)$ is a polynomial in 
$\lambda$ of degree $2n$. Let $\left<\cdot,\cdot\right>$ be the bi-linear form
defined by 
\eqref{killnorm}. Let $\Lambda(\sigma)\in\mathfrak{b}_{\mathbb{C}}^{*}$ be 
the highest weight of $\sigma$ as in \eqref{Darstellungen von M}.
Then by theorem 13.2 in \cite{Kn1} there exists a constant $c(n)$ such that 
one has
\begin{align*}
\mu_{\sigma}(\lambda)=-c(n)\prod_{\alpha\in\Delta^{+}(\mathfrak{g}_{\mathbb{C}},
\mathfrak{h}_{\mathbb{C}})}\frac{\left<i\lambda e_{1}+\Lambda(\sigma)+\rho_{M},
\alpha\right>}{\left<\rho_{G},\alpha\right>}..
\end{align*}
The constant $c(n)$ is computed in \cite{Mi2}. By \cite{Mi2}, theorem 3.1, one 
has $c(n)>0$. 
For $z\in\mathbb{C}$  let
\begin{align}\label{plancherelmass}
P_{\sigma}(z)=-c(n)\prod_{\alpha\in\Delta^{+}(\mathfrak{g}_{\mathbb{C}},
\mathfrak{h}_{\mathbb{C}})}\frac{\left<z e_{1}+\Lambda(\sigma)+\rho_{M},
\alpha\right>}{\left<\rho_{G},\alpha\right>}.
\end{align}
One easily sees that
\begin{align}\label{P-Polynom is W-inv}
P_{\sigma}(z)=&P_{w_{0}\sigma}(z).
\end{align}

\section{The regularized trace}\label{sec-regtrace}

Regard $G$ as a principal $K$-fibre bundle over $\tilde{X}$. By the invariance
 of $\mathfrak{p}$ under $\Ad(K)$, the assignment
\begin{align*}
T_{g}^{{\rm{hor}}}:=\{\frac{d}{dt}\bigr|_{t=0}g\exp{tX}\colon
X\in\mathfrak{p}\} 
\end{align*}
defines a horizontal distribution on $G$. This connection is called the 
canonical connection. 
Let $\nu$ be a finite-dimensional unitary representation of $K$ on 
$(V_{\nu},\left<\cdot,\cdot\right>_{\nu})$. Let
\begin{align*}
\tilde{E}_{\nu}:=G\times_{\nu}V_{\nu}
\end{align*}
be the associated homogeneous vector bundle over $\tilde{X}$. Then 
$\left<\cdot,\cdot\right>_{\nu}$ induces a $G$-invariant metric 
$\tilde{B}_{\nu}$ on $\tilde{E}_{\nu}$. Let $\widetilde{\nabla}^{\nu}$ be the 
connection on $\tilde{E}_{\nu}$ induced by the canonical connection. Then 
$\widetilde{\nabla}^{\nu}$ is $G$-invariant.
Let  
\begin{align*}
E_{\nu}:=\Gamma\backslash(G\times_{\nu}V_{\nu})
\end{align*}
be the associated locally homogeneous bundle over $X$. Since 
$\tilde{B}_{\nu}$ and $\widetilde{\nabla}^{\nu}$ are $G$-invariant, they push 
down to a metric $B_{\nu}$ and a connection $\nabla^{\nu}$ on $E_{\nu}$. Let
\begin{align}\label{globsect}
C^{\infty}(G,\nu):=\{f:G\rightarrow V_{\nu}\colon f\in C^\infty,\;
f(gk)=\nu(k^{-1})f(g),\,\,\forall g\in G, \,\forall k\in K\}.
\end{align}
Let
\begin{align}\label{globsect1}
C^{\infty}(\Gamma\backslash G,\nu):=\left\{f\in C^{\infty}(G,\nu)\colon 
f(\gamma g)=f(g)\:\forall g\in G, \forall \gamma\in\Gamma\right\}.
\end{align}
Let $C^{\infty}(\widetilde X,\widetilde E_{\nu})$ (resp. $C^{\infty}(X,E_{\nu})$)
denote the space of smooth sections of $\widetilde E_{\nu}$ (resp. $E_\nu$). 
Then there are canonical isomorphisms
\begin{equation}\label{isomor1}
\widetilde \phi:C^{\infty}(\widetilde X,\widetilde E_{\nu})\cong 
C^{\infty}(G,\nu)\quad\text{and}\quad 
\phi:C^{\infty}(X,E_{\nu})\cong C^{\infty}(\Gamma\backslash G,\nu)
\end{equation}
(see \cite[p. 4]{Mia}).
There are also corresponding isometries for the spaces 
$L^{2}(\widetilde X,\widetilde E_{\nu})$ and $L^{2}(X,E_{\nu})$ of 
$L^{2}$-sections of $\widetilde E_{\nu}$ and $E_{\nu}$, respectively. For every 
$X\in\mathfrak{g}$, $g\in G$, and every $f\in C^{\infty}(X,E_{\nu})$ one has
\begin{align*}
\phi(\nabla^{\nu}_{L(g)_{*}X}f)(g)=\frac{d}{dt}\phi(f)(g\exp{tX})\big|_{t=0}.
\end{align*}
Let
\[
\Delta_\nu=(\nabla^\nu)^*\nabla^\nu
\]
be the Bochner-Laplace operator acting in $C^\infty(X,E_\nu)$.  Since $X$ is 
complete, $\Delta_\nu$ regarded
as linear operator in $L^2(X,E_\nu)$ with domain $C^\infty_c(X,E_\nu)$ is 
essentially self-adjoint \cite{Che}. By \cite[Proposition1.1]{Mia} it follows 
that on $C^{\infty}(\Gamma\bs G,\nu)$ one has
\begin{align}\label{BLO}
\Delta_{\nu}=-R_\Gamma(\Omega)+\nu(\Omega_K),
\end{align} 
where $R_\Gamma$ denotes the right regular representation of $G$ on 
$C^\infty(\Gamma\bs G,\nu)$. Let $A_\nu$ be the differential operator in
$C^\infty(X,E_\nu)$ which acts as $-R_\Gamma(\Omega)$ in 
$C^\infty(\Gamma\bs G,\nu)$. If $\nu$ is irreducible, then $\nu(\Omega_K)$ is
a scalar. In general it is an endomorphism of $E_\nu$ which commutes with 
$A_\nu$. It follows from \eqref{BLO} that $A_\nu$ is a self-adjoint operator 
which is bounded from below. Therefore, the heat operator $e^{-tA_\nu}$ is well
defined and we have
\begin{equation}\label{heat-bl}
e^{-t\Delta_\nu}=e^{-t\nu(\Omega_K)} e^{-tA_\nu}.
\end{equation}
 Let $H^\nu(t,x,y)$ be the kernel of $e^{-tA_\nu}$. Let $X(Y)\subset X$ be 
defined by \eqref{truncmfd}. For $Y\gg 0$ this is compact manifold with 
boundary. It follows from \cite[(5.6)]{MP2} that there exist smooth 
functions $a(t)$ and $b(t)$ such that
\[
\int_{X(Y)}\tr H^\nu(t,x,x)\;dx=a(t)\log Y+b(t)+o(1)
\]
as $Y\to\infty$. Put
\begin{equation}\label{regtrace}
\Tr_{\reg}\left(e^{-tA_\nu}\right):=b(t).
\end{equation}
In \cite[(5.6)]{MP2}, the regularized trace $\Tr_{\reg}(e^{-tA_\nu})$ is  
described explicitly in terms of the discrete spectrum of $A_\nu$ and the 
intertwining operators. To state the formula we need to introduce some notation.
Let $\sigma\in\widehat M$ with highest weight given by 
\eqref{Darstellungen von M}. Let $\rho_j$, $j=1,\dots,n+1$ be defined by
\eqref{rho}. Put
\begin{equation}\label{csigma}
c(\sigma):=\sigma(\Omega_M)-n^2=\sum_{j=2}^{n+1}(k_j(\sigma)+\rho_j)^2-
\sum_{j=1}^{n+1}\rho_j^2,
\end{equation}
where the second equality follows from a standard computation. 
Let $w\in W(A)$ be the nontrivial element.
For $\lambda\in\C$ let
\begin{equation}\label{intertwop2}
\mathbf{C}(\sigma,\lambda)\colon \boldsymbol{\cE}(\sigma)\to
\boldsymbol{\cE}(w\sigma)
\end{equation}
be the intertwining operator defined in \cite[(3.13)]{MP2}. Let 
\[
\widetilde{\boldsymbol{C}}(\sigma,\nu,\lambda)\colon  
(\boldsymbol{\cE}(\sigma)\otimes V_\nu)^K\to 
(\boldsymbol{\cE}(w\sigma)\otimes V_\nu)^K
\]
be the restriction of
$\boldsymbol{C}(\sigma,\lambda)\otimes \Id_{V_\nu}$ to 
$(\boldsymbol{\cE}(\sigma)\otimes V_\nu)^K$. Furthermore, let
$A_\nu^d$ denote the restriction of $A_\nu$ to the discrete subspace
$L^2_d(X,E_\nu)$ of $A_\nu$. Then by \cite[(5.6)]{MP2} we have
\begin{equation}\label{regtrace2}
\begin{split}
\Tr_{\reg}\left(e^{-tA_\nu}\right)&=\Tr\left(e^{-tA^d_\nu}\right)+
\sum_{\substack{\sigma\in\hat{M};\sigma=w_0\sigma\\
\left[\nu:\sigma\right]\neq
0}}e^{tc(\sigma)}
\frac{\Tr(\widetilde{\boldsymbol{C}}(\sigma,\nu,0))}{4}\\
&-\frac{1}{4\pi}\sum_{\substack{\sigma\in\hat{M}\\
\left[\nu:\sigma\right]\neq 0}}e^{tc(\sigma)}\int_{\R}e^{-t\lambda^2}
\Tr\left(\widetilde{\boldsymbol{C}}(\sigma,\nu,-i\lambda)
\frac{d}{dz}\widetilde{\boldsymbol{C}}(\sigma,\nu,i\lambda)\right)\,d\lambda.
\end{split}
\end{equation}
It follows from \cite[Theorem 8.4]{Wa1} that the right hand side of 
\eqref{regtrace2} equals the spectral side of the Selberg trace formula
applied to $e^{-tA_\nu}$. 

\section{The trace formula}\label{sec-trform}
\setcounter{equation}{0}

In this section we apply the Selberg trace formula to study the regularized
trace of the heat operator $e^{-tA_\nu}$. To this end we 
briefly recall the Selberg trace formula. First we introduce the
distributions on the geometric side which are associated to the different
conjugacy classes of $\Gamma$.  
Let $\alpha$ be a K-finite Schwartz function on $G$. The contribution of the
identity is 
\begin{align*}
I(\alpha):=\vol(\Gamma\backslash G)\alpha(1).
\end{align*}
By \cite[Theorem 3]{HC}, the Plancherel theorem can be applied to $\alpha$. For
groups of real rank one which do not possess a compact Cartan subgroup it is
stated in \cite[Theorem 13.2]{Kn1}. For $\sigma\in\widehat M$ and $\lambda\in\C$
let $\pi_{\sigma,\lambda}$ be the principle series representation which we 
parametrize as in \cite[Sect. 2.7]{MP2}. Let $\Theta_{\sigma,\lambda}$ be the 
character of $\pi_{\sigma,\lambda}$. Let $P_\sigma(z)$ be the polynomial defined by 
\eqref{plancherelmass}. Then one has
\begin{align}\label{Idcontr}
I(\alpha)=\vol(X)\sum_{\substack{\sigma\in\hat{M}\\\left[\nu:\sigma\right]\neq
0}}\int_{\mathbb{R}}{P_{\sigma}(i\lambda)\Theta_{\sigma,\lambda}(\alpha)}
d\lambda,
\end{align}
where the sum is finite since $\alpha$ is $K$-finite. In even dimensions an
additional contribution of the discrete series appears.

Next let $\CC(\Gamma)_{\s}$ be the set of semi-simple conjugacy classes
$\left[\gamma\right]$. The contribution of the hyperbolic conjugacy classes
is given by
\begin{align*}
H(\alpha):=\int_{\Gamma\backslash
G}\sum_{\left[\gamma\right]\in\CC(\Gamma)_{\s}-\left[1\right]}
\alpha(x^{-1}\gamma x)dx.
\end{align*}
By \cite[Lemma 8.1]{Wa1} the integral converges absolutely. 
Its Fourier
transform can be computed as
follows.
Since $\Gamma$ is assumed to be torsion free, every
nontrivial semi-simple element $\gamma$ is conjugate to an element
$m(\gamma)\exp{\ell(\gamma)H_1}$, $m(\gamma)\in M$. By \cite[Lemma
6.6]{Wal},
$l(\gamma)>0$
is unique and $m(\gamma)$ is determined up to conjugacy in $M$.
Moreover, $\ell(\gamma)$ is the length of the unique closed geodesic
associated to $\left[\gamma\right]$. It follows that $\Gamma_{\gamma}$, the
centralizer of $\gamma$ in $\Gamma$, is infinite cyclic. Let $\gamma_{0}$ denote
its generator which is semi-simple too. For
$\gamma\in\left[\Gamma\right]_{S}-\{\left[1\right]\}$ let
$a_\gamma:=\exp{\ell(\gamma)H_1}$ and let
\begin{align}\label{hyperbcontr}
L(\gamma,\sigma):=\frac{\overline{\Tr(\sigma)(m_{\gamma})}}
{\det\left(\Id-\Ad(m_\gamma a_\gamma)|_{\bar\nf}\right)}e^{-n\ell(\gamma)}.
\end{align}
Proceeding as in \cite{Wal} and using \cite[equation 4.6]{Ga}, one obtains
\begin{align}\label{Hyperb}
H(\alpha)=&\sum_{\sigma\in\hat{M}}\sum_{\left[\gamma\right]\in \CC(\Gamma)_{\s}
-\left[1\right]}\frac{l(\gamma_{0})}{2\pi}
L(\gamma,\sigma)\int_{-\infty}^{\infty}{\Theta_{\sigma,\lambda}(\alpha)e^{
-il(\gamma)\lambda}d\lambda},
\end{align}
where the sum is finite since $\alpha$ is $K$-finite.  

Next we describe the parabolic contribution. Put
\begin{equation}
T(\alpha):=\int_{K}\int_{N}{\alpha(kn k^{-1})dn}.
\end{equation}
Let $n\in N$. There exists a unique $Y\in\nf$ such that $n=\exp(Y)$. Put
$\|n\|:=\|Y\|$. Then let
\begin{equation}\label{weigh-orb}
T'(\alpha):=\int_{K}\int_{N}\alpha(kn k^{-1})\log\|n\|\;dn\;dk.
\end{equation}
By the Theorem on p. 299 in \cite{OW} there exist constants $C_1(\Gamma)$ and 
$C_2(\Gamma)$ such
that the contribution of the parabolic conjugacy classes equals
\begin{equation}\label{parabol-contr}
C_1(\Gamma)T(\alpha)+C_2(\Gamma)T^\prime(\alpha).
\end{equation}
The distributions $T$ and $T^\prime$ are tempered and $T$ is an invariant
distribution. Applying the Fourier inversion formula and
the Peter-Weyl-Theorem to
equation 10.21 in \cite{Kn1}, one obtains 
the Fourier transform of T as:
\begin{align}\label{Fouriertrafo T}
T(\alpha)=\sum_{\sigma\in\hat{M}}\dim(\sigma)\frac{1}{2\pi}
\int_{\mathbb{R}}\Theta_{\sigma,\lambda}(\alpha)d\lambda.
\end{align}
The distribution $T^\prime_P$ is not invariant. One way to deal with this
distribution is to make it invariant (see \cite[(6.15)]{MP2}) and then
apply the Fourier inversion formula of \cite{Ho2}. As explained in the
introduction, we will use a different method.

Let $\widetilde A_\nu$ be the differential operator in $C^\infty(\widetilde X,
\widetilde E_\nu)$ induced by $-\Omega$. This is the lift of $A_\nu$ to 
$\widetilde X$. Let $\widetilde \Delta_\nu=(\widetilde \nabla^\nu)^*\widetilde
\nabla^\nu$ be the Bochner-Laplace operator associated to the canonical 
connection in $\widetilde E_\nu$. Then we have
\begin{equation}\label{BLO1}
\widetilde \Delta_\nu=\widetilde A_\nu+\nu(\Omega_K).
\end{equation}
Denote by $\widetilde H^\nu(t,x,y)$ the kernel of the heat operator 
$e^{-t\widetilde A_\nu}$. 
Observe that $\widetilde H^\nu(t,x,y)\in
\Hom((\widetilde E_\nu)_y,(\widetilde E_\nu)_x)$. 
For $g\in G$ and $x\in\widetilde X$ let $L_g\colon \widetilde E_x\to \widetilde
E_{gx}$ be the isomorphism induced by the left translation. Since $\widetilde
\Delta_\nu$ commutes with the action of $G$, the kernel satisfies 
\begin{equation}\label{isomor2}
L_g^{-1}\circ \widetilde H^\nu(t,gx,gy)\circ L_g
=\widetilde H^\nu(t,x,y),\quad x,y\in\widetilde X,\,
g\in G,
\end{equation}
considered as a linear map $\widetilde E_y\to \widetilde E_x$.
Let $x_0:=eK\in\widetilde X$. We identify $\widetilde E_{x_0}$ with $V_\nu$. 
Using the isomorphism \eqref{isomor1}, $\widetilde H^\nu(t,x,y)$ corresponds 
to a kernel
\[
\widetilde H^\nu_t\colon G\times G\to \End(V_\nu),
\]
which is defined by
\begin{equation}\label{kernel6}
\widetilde H^\nu_t(g_1,g_2):=L_{g_1}^{-1}\circ \widetilde H^\nu(t,g_1x_0,g_2x_0)
\circ L_{g_2}.
\end{equation}
By \eqref{isomor2} it follows that it satisfies
\begin{equation}\label{invar}
\widetilde H^\nu_t(gg_1,gg_2)=\widetilde H^\nu_t(g_1,g_2),\quad g,g_1,g_2\in G,
\end{equation}
and
\begin{equation}\label{transf1}
\widetilde H^\nu_t(g_1k_1,g_2k_2)=\nu(k_1)^{-1}\circ \widetilde 
H^\nu_t(g_1,g_2)\circ \nu(k_2),\quad k_1,k_2\in K,\; g\in G.
\end{equation}
Using \eqref{invar}, we can identify $\widetilde H^\nu_t$ with a map 
\[
H^\nu_t\colon G\to\End(V_\nu)
\]
by
\begin{equation}\label{convol}
H^\nu_t(g):=\widetilde H^\nu_t(e,g),\quad g\in G.
\end{equation}
Then $H^\nu_t$ belongs to $(\Co^1(G)\otimes \End(V_\nu))^K$ and satisfies
\begin{equation}\label{transf2}
H^\nu_t(k_1gk_2)=\nu(k_1)\circ H^\nu_t(g)\circ\nu(k_2),\quad k_1,k_2\in K,\;
g\in G.
\end{equation}
Let $h^\nu_t$ be defined by
\begin{equation}\label{trace4}
h_t^\nu(g):=\tr H^\nu_t(g),\quad g\in G.
\end{equation}
Then belongs $h^\nu_t$ to $\Co^1(G)$ (see \cite{BM}). If we apply the Selberg
trace formula \cite[Theorem 8.4]{Wa1} to \eqref{regtrace2} and use
\eqref{parabol-contr}, we obtain the following theorem.
\begin{theo}\label{regtrace3}
For all $t>0$ we have
\[
\Tr_{reg}\left(e^{-tA_\nu}\right)=I(h^\nu_t)+H(h^\nu_t)+C_1(\Gamma) T(h^\nu_t)
+C_2(\Gamma)T^\prime(h^\nu_t).
\] 
\end{theo}
This theorem can be used to determine the asymptotic behavior of 
$\Tr_{reg}\left(e^{-tA_\nu}\right)$ as $t\to 0$. For $I(h^\nu_t)$ we use
\eqref{Idcontr}. The character $\Theta_{\sigma,\lambda}(h_t^\nu)$ is computed by
\cite[Proposition 4.1]{MP2}. We have
\begin{equation}\label{character}
\Theta_{\sigma,\lambda}(h_t^\nu)=e^{t(c(\sigma)-\lambda^2)}.
\end{equation}
By \eqref{Idcontr} it follows that
\[
I(h^\nu_t)=\vol(X)
\sum_{\substack{\sigma\in\hat{M}\\\left[\nu:\sigma\right]\neq0}}e^{tc(\sigma)}
\int_{\R}e^{-t\lambda^2}P_{\sigma}(i\lambda)\;d\lambda.
\]
Now recall that $P_\sigma(z)$ is an even polynomial of degree $2n=d-1$.
Hence we obtain an expansion
\begin{equation}\label{Idcontr1}
I(h^\nu_t)=t^{-d/2}\sum_{j=0}^\infty a_jt^j.
\end{equation}
Using \eqref{character} and \eqref{Hyperb}, it follows that
\begin{equation}\label{hyper-asymp}
H(h_t^\nu)=O(e^{-c/t}),\quad 0<t\le 1.
\end{equation}
Furthermore, by \eqref{Fouriertrafo T} and \eqref{character} we get
\begin{equation}\label{parab-asymp}
T(h_t^\nu)=t^{-1/2}\sum_{j=0}^\infty b_j t^j.
\end{equation}
It remains to determine the asymptotic behavior of $T^\prime(h_t^\nu)$. This
will be done in the next sections.

\section{Heat kernel estimates}\label{sec-heatker}
\setcounter{equation}{0}

In this section we study the kernel $K^\nu(t,x,y)$ of $e^{-t\Delta_\nu}$. 
Observe that 
$K^\nu(t,x,y)\in\Hom((\widetilde E_\nu)_y,(\widetilde E_\nu)_x)$. Denote
by $|K^\nu(t,x,y)|$ the norm of this homomorphism.
Furthermore, let $r(x,y)$ denote the geodesic distance of $x,y\in\widetilde X$.
We have
\begin{prop}\label{prop-estim}
For every $T>0$ there exists $C>0$ such that for all $\nu\in\widehat K$ we have
\[
|K^\nu(t,x,y)|\le C t^{-d/2}  exp\left(-\frac{r^2(x,y)}{4t}\right)
\]
for all $0<t\le T$ and $x,y\in\widetilde X$, where $d=\dim\widetilde X$.
\end{prop}
\begin{proof}
If $\nu$ is irreducible, this is proved in \cite[Proposition 3.2]{Mu1}. 
However, the proof does not make any use of the irreducibility of $\nu$. 
So it extends without any change to the case of finite-dimensional 
representations.
\end{proof}
By \eqref{BLO1} the kernel $\widetilde H^\nu(t,x,y)$ of $e^{-tA_\nu}$ 
is closely related to the kernel $K^\nu(t,x,y)$. If $\nu$ is irreducible, 
$\nu(\Omega_K)$ is a scalar and we have
\begin{equation}\label{kernel-rela}
\widetilde H^\nu(t,x,y)=e^{t\nu(\Omega_K)}K^\nu(t,x,y).
\end{equation}
Let $h_t^\nu\in \mathcal{C}^1(G)$ be defined by \eqref{trace4}.
Note that for each $g\in G$, $L_g\colon E_x\to E_{gx}$ is an isometry.
Thus using \eqref{kernel6}, the definition of $H^\nu_t$, 
\eqref{kernel-rela}, and Proposition \ref{prop-estim}, we get
\begin{corollary}\label{estim3}
Let $d=\dim\widetilde X$. For all $T>0$ there exists $C>0$ such that we have
\[
|h_t^\nu(g)|\le C  t^{-d/2} \exp\left(-\frac{r^2(gx_0,x_0)}{4t}\right)
\]
for all $0<t\le T$ and $g\in G$.
\end{corollary}
Next we turn to the asymptotic expansion of the heat kernel. Let 
$d_{\bx}\exp_{x_0}$ be the differential of the exponential map $\exp_{x_0}\colon
T_{x_0}\widetilde X\to\widetilde X$ at the point 
$\bx\in T_{x_0}\widetilde X$. It is a map
from $T_{x_0}\widetilde X$ to $T_x\widetilde X$, where $x=\exp_{x_0}(\bx)$. Let
\begin{equation}\label{jacob1}
j(\bx):=|\det(d_{\bx}\exp_{x_0})|
\end{equation}
be the Jacobian, taken with respect to the inner products in the tangent spaces.
Note that
\begin{equation}\label{jacob2}
j(\bx)=|\det(g_{ij}(\bx))|^{1/2}.
\end{equation}
Write $y=\exp_x(\by)$, with $\by\in T_x\tilde X$. Let $\ve>0$ be sufficiently
small. Let $\psi\in C^\infty(\R)$ with $\psi(u)=1$ for $u<\ve$ and 
$\psi(u)=0$ for $u>2\ve$.

\begin{prop}\label{prop-asympexp}
Let $d=\dim\widetilde X$. Let $(\nu,V_\nu)$
be a finite-dimensional unitary representation of $K$. There
exist smooth sections $\Phi_i^\nu\in C^\infty(\widetilde X\times\tilde X,
\widetilde E_\nu\boxtimes \widetilde E^*_\nu)$, $i\in\N_0$, such that for 
every $N\in\N$
\begin{equation}\label{asympexp}
\begin{split}
K^\nu(t,x,y)=t^{-d/2}\psi(d(x,y))\exp\left(-\frac{r^2(x,y)}{4t}\right)
\sum_{i=0}^N\Phi_i^\nu(x,y&) j(x,y)^{-1/2}t^i\\
&+O(t^{N+1-d/2}),
\end{split}
\end{equation}
uniformly for  $0< t\le 1$. Moreover the leading term $\Phi_0^\nu(x,y)$ is
equal to the parallel transport $\tau(x,y)\colon (\widetilde E_\nu)_y\to 
(\widetilde E_\nu)_x$ with respect to the connection $\nabla^\nu$ along the 
unique geodesic joining $x$ and $y$. 
\end{prop}
\begin{proof}
Let $\Gamma\subset G$ be a co-compact torsion free lattice. It exists by
\cite{Bo}. Let $X=\Gamma\bs \widetilde X$ and $E_\nu=\Gamma\bs 
\widetilde E_\nu$. As in \cite[Sect 3]{Do}, the proof can be reduced to the 
compact case, which follows from \cite[Theorem 2.30]{BGV}. 
\end{proof}

Let $\pf$ be as in \eqref{cartan}. We recall that the mapping 
\[
\varphi\colon  \pf\times K\to G,
\]
defined by $\varphi(Y,k)=\exp(Y)\cdot k$ is a diffeomorphism 
\cite[Ch. VI, Theorem 1.1]{He}. Thus each $g\in G$ can be uniquely
written as 
\begin{equation}\label{cartan2}
g=\exp(Y(g))\cdot k(g),\quad Y(g)\in\pf,\; k(g)\in K.
\end{equation} 
Using Proposition \ref{prop-asympexp} and \eqref{kernel-rela}, we obtain the 
following corollary.
\begin{corollary}
There exist $a^\nu_i\in C^\infty(G)$ such that
\begin{equation}\label{asympexp1}
h^\nu_t(g)=t^{-d/2}\psi(d(gx_0,x_0))
\exp\left(-\frac{r^2(gx_0,x_0)}{4t}\right)
\sum_{i=0}^N a_i^\nu(g) t^i+O(t^{N+1-d/2})
\end{equation}
which holds for  $0<t\le1$. Moreover the leading 
coefficient $a_0^\nu$ is given by
\begin{equation}\label{leadcoeff}
a_0^\nu(g)=\tr(\nu(k(g)))\cdot j(x_0,gx_0)^{-1/2}.
\end{equation}
\end{corollary}
\begin{proof}
By \eqref{kernel6} and \eqref{convol} we have
\[
H^\nu_t(g)=H^\nu(t,x_0,gx_0)\circ L_g,\quad g\in G.
\]
Put
\begin{equation}\label{trace5}
a_i^\nu(g):=\tr(\Phi_i^\nu(x_0,gx_0)\circ L_g)\cdot j(x_0,gx_0)^{-1/2},
\quad g\in G.
\end{equation}
Then \eqref{asympexp1} follows immediately from \eqref{asympexp} and the 
definition of $h_t^\nu$. To prove the second statement, we recall that
$\Phi_0^\nu(x,y)$ is the parallel transport $\tau(x,y)$ 
with respect to the canonical
connection of $\widetilde E_\nu$ along the geodesic connecting $x$ and $y$.
Let $g=\exp(Y)\cdot k$, $Y\in\pf$, $k\in K$. Then the geodesic connecting 
$x_0$ and $gx_0$ is the curve $\gamma(t)=\exp(tY)x_0$, $t\in [0,1]$
(see \cite[Ch. IV, Theorem 3.3]{He}). The parallel transport along $\gamma(t)$
equals $L_{\exp(Y)}$. Thus $\Phi_0^\nu(x_0,gx_0)=L_{\exp(Y)}^{-1}$. Hence we get.
\[
\Phi_0^\nu(x_0,gx_0)\circ L_g=L_k=\nu(k).
\]
Together with \eqref{trace5} the claim follows.
\end{proof}

\section{Weighted orbital integrals}\label{sec-weightorb}
\setcounter{equation}{0}

The weighted orbital integral is given by \eqref{weigh-orb}. 
We apply this to $h^\nu_t$. By \eqref{transf2} it follows that $h^\nu_t$ is 
invariant under conjugation by $k\in K$. Thus we get
\begin{equation}\label{weigh-orb1}
T^\prime_P(h^\nu_t)=\int_Nh^\nu_t(n)\log\|n\|\,dn.
\end{equation}
We fix an isometric identification  of $\R^{d-1}$ with $\nf$ with respect to
the inner product $\langle\cdot,\cdot\rangle_\theta$ defined by
\[
\langle Y_1,Y_2 \rangle_\theta:=\langle Y_1,\theta(Y_2)\rangle,\quad 
Y_1,Y_2\in\gf.
\]
Explicitly it is given by
\begin{equation}\label{idendfic1}
x\in\R^{d-1}\mapsto Y(x):=\begin{pmatrix}0&-x&x\\x^T&0&0\\x^T&0&0\end{pmatrix},
\end{equation}
where we consider $x$ as a column. Furthermore we identify $\nf$
with $N$ via the exponential map. Put
\[
n(x):=\exp(Y(x))\in N,\quad x\in\R^{d-1}.
\]
We note that 
\begin{equation}\label{exponen}
n(x)=\Id+Y(x)+\frac{1}{2}Y(x)^2.
\end{equation}
Then we get
\begin{equation}\label{weightint}
T^\prime(h^\nu_t)=\int_{\R^{d-1}}h^\nu_t\left(n(x)\right)\log\|x\| \,dx.
\end{equation}
For $\ve>0$ let $B(\ve)\subset \R^{d-1}$ denote the ball of
radius $\ve$ centered at $0$ and let
\[
U(\ve):= \R^{d-1}-B(\ve).
\]
We decompose the integral as 
\[
\int_{\R^{d-1}}=\int_{B(\ve)}+\int_{U(\ve)}.
\]
Put
\begin{equation}\label{dist3}
r(x):=r(n(x)x_0,x_0),\quad x\in\R^{d-1}.
\end{equation}
We need some properties of this function.
\begin{lem}
We have
\begin{equation}\label{dist2}
r(x)=\arcosh\left(1+\frac{\|x\|^2}{2}\right),\quad x\in\R^{d-1}.
\end{equation}
\end{lem}
\begin{proof}
Let $t(x)\ge 0$ be the radial component of $n(x)$, defined by \eqref{radcomp}.
By \cite[Lemma 7.1]{Wa1} we have
\[
t(x)=\arcosh\left(1+\frac{\|x\|^2}{2}\right).
\]
The lemma follows from \eqref{dist}.
\end{proof}
Now note that
\[
\arcosh(x)=\ln\left(x+\sqrt{x^2-1}\right),\quad x\ge 1.
\]
Thus
\begin{equation}\label{log-est}
r(x)\ge \ln\left(1+\frac{\|x\|^2}{2}\right),\quad x\in\R^{d-1}.
\end{equation}
This implies that for all $t>0$ we have
\begin{equation}\label{int1}
\int_{\R^{d-1}}\exp\left(-\frac{r^2(x)}{4t}\right)|\log\|x\||\,dx <\infty.
\end{equation}
Using $\arcosh^\prime(x)=(x^2-1)^{-1/2}$, we get
\[
\frac{d}{dr}\arcosh\left(1+\frac{r^2}{2}\right)=\frac{1}{\sqrt{1+r^2/4}},\quad
r\ge 0.
\]
Thus
\begin{equation}\label{inequal}
r(x_1)>r(x_2), \quad\text{if}\;\;\|x_1\|>\|x_2\|.
\end{equation}

Now observe that $\arcosh(1+x^2/2)$ has a Taylor series expansion of the form
\[
\arcosh\left(1+\frac{x^2}{2}\right)=x+\sum_{k=1}^\infty a_k x^{2k+1}
\]
which converges for $|x|<1/2$. This follows from \cite[1.631.2]{GR} together
\cite[1.641.2]{GR}. Thus $r^2(x)$ is $C^\infty$ and we have
\begin{equation}\label{distsq}
r^2(x)=\|x\|^2+R(x)
\end{equation} 
with 
\[
|R(x)|\le C\|x\|^4,\quad \|x\|\le 1/4.
\]
Thus $r^2(x)$ has a non-degenerate critical point at $x=0$ of index $(d-1,0)$.

Now we turn to the estimation of the orbital integral. Put
\[
c(\varepsilon):=\arcosh(1+\varepsilon^2/2).
\]
By Corollary \ref{estim3}, \eqref{int1} and \eqref{inequal} we get 
\begin{equation}\label{estima5}
\begin{split}
\left|\int_{U(\ve)}h^{\nu}_t(n(x))\log\|x\|\,dx\right|&\le C t^{-d/2} 
\int_{U(\ve)}\exp\left(-\frac{r^2(x)}{4t}\right)|\log\|x\||
dx\\
&\le C_1t^{-d/2}\exp\left(-\frac{c(\varepsilon)}{8t}\right)
\end{split}
\end{equation}
for $0< t\le 1$. To deal with the integral over $B(\ve)$, we use 
\eqref{asympexp1}. This gives
\begin{equation}\label{asympexp2}
\begin{split}
\int_{B(\ve)}h^\nu_t(n(x))\log\|x\|\,dx=&t^{-d/2}\sum_{i=1}^N 
t^i\int_{B(\ve)}\exp\left(-\frac{r^2(x)}{4t}\right)\psi(x)a_i^\nu(x)
\log\parallel x\parallel dx\\
&+O(t^{N+1-d/2})
\end{split}
\end{equation}
for $0< t\le 1$, where $\psi\in C^\infty_c(B(\ve))$. Put $m=d-1$. The integrals 
on the right hand side are of the form
\begin{equation}\label{integr4}
I(\lambda)=\int_{\R^m}e^{-\lambda f(x)}g(x)\log\parallel x\parallel dx,
\end{equation}
where $\lambda>0$, $g\in C^\infty_c(\R^m)$, $\supp g\subset B(\ve)$, 
and $f$  satisfies
\begin{equation}\label{bound}
f(x)=\|x\|^2+R(x),\quad |R(x)|<C\|x\|^4,\;\;\|x\|<\ve,
\end{equation}
and $f$ has no critical points in $\supp g\setminus\{0\}$.
Our goal is to derive an 
asymptotic expansion for $I(\lambda)$ as $\lambda\to \infty$. To begin with,  
we first 
show that $f$ can be replaced by $\|x\|^2$. We proceed as in H\"ormander's
 proof of the stationary phase approximation \cite[Theorem 7.7.5]{Hor}.
On $B(\ve)\setminus\{0\}$ we consider the following differential operator
\begin{equation}\label{diffop1}
L:=-\frac{1}{\|f^\prime(x)\|^2}\sum_{j=1}^m\frac{\partial f}{\partial x_j}
\frac{\partial}{\partial x_j}.
\end{equation}
Note that the formally adjoint operator $L^*$ is given by
\begin{equation}\label{diffop2}
L^*=\sum_{j=1}^m\frac{\partial}{\partial x_j}\frac{1}{\|f^\prime(x)\|^2}
\frac{\partial f}{\partial x_j},
\end{equation}
where the factors act as multiplication operators. Since integration by parts
arguments will introduce singularities, we make some more general assumptions.
Suppose that $g\in C^1(\R^m\setminus\{0\})$ with support in $B(\ve)$ and with
$|D^\alpha g|\le C\|x\|^{3-|\alpha|}$ for $|\alpha|\le 1$. Using the divergence
theorem in the last step, we have
\[
\begin{split}
I(\lambda)&=\lim_{r\to 0}\int_{\|x\|\ge r}e^{-\lambda f(x)}g(x)\log\|x\|\;dx=
\lambda^{-1}\lim_{r\to 0}\int_{\|x\|\ge r}(Le^{-\lambda f(\cdot)})(x)g(x)\log\|x\|\;dx\\
&=\lambda^{-1}\lim_{r\to 0}\int_{\|x\|\ge r}e^{-\lambda f(x)}L^*(g\log\|\cdot\|)(x)\;dx
\\
& \quad+\lambda^{-1}\lim_{r\to 0}\int_{\|x\|=r}\|f^\prime(x)\|^{-2}
\langle\nu,\nabla f\rangle(x) e^{-\lambda f(x)}g(x)\log\|x\|\;dS(x),
\end{split}
\]
where $\nu$ is the Durward unit normal vector field to 
$\partial(\R^m\setminus B(r))$. By the assumption on $f$, there exists $C>0$
such that
\[
\|f^\prime(x)\|^{-2}\le C\|x\|^{-2},\quad \|\nabla f(x)\|\le C\|x\|.
\]
Together with the assumptions on $g$, it follows that the integrand in the
surface integral is bounded on $B(\ve)$. Thus the surface integral has limit
$0$ as $r\to 0$. Furthermore, by \eqref{diffop2} we have
\[
L^*(g\log\|\cdot\|)(x)=\sum_{j=1}^m\frac{\partial}{\partial x_j}
\left(\frac{1}{\|f^\prime(x)\|^2}\frac{\partial f}{\partial x_j}(x)g(x)\log\|x\|
\right).
\]
Using the assumptions on $f$ and $g$, it follows that $L^*(g\log\|\cdot\|)$ is 
bounded, and therefore 
\[
I(\lambda)=\lambda^{-1}\lim_{r\to 0}\int_{\|x\|\ge r}e^{-\lambda f(x)}
L^*(g\log\|\cdot\|)(x)\;dx=\lambda^{-1}\int_{\R^m}e^{-\lambda f(x)}
L^*(g\log\|\cdot\|)(x)\;dx.
\]
Using the properties of $f$ and $g$, we get
\[
|I(\lambda)|\le C(f)\lambda^{-1}\sum_{|\alpha|\le 1}
\sup_{x\in \R^m}|D^\alpha (g(x)\log\|x\|)\cdot\|x\|^{-2+|\alpha|}|.
\]
Now let $k\in\N$ and assume that $g\in C^k(\R^m\setminus\{0\})$ has support in
$B(\ve)$ and satisfies 
\[
|D^\alpha g(x)|\le C \|x\|^{2k+1-|\alpha|},\quad \text{for}\quad|\alpha|\le k.
\]
Let $u(x)=g(x)\log\|x\|$, $x\in\R^m\setminus\{0\}$. Then $u\in 
C^k(\R^m\setminus\{0\})$, with support in
$B(\ve)$ and with $|D^\alpha u(x)|\le C \|x\|^{2k-|\alpha|}$ for $|\alpha|\le k$.
Then it follows that $L^*(u)\in C^{k-1}(\R^m\setminus\{0\})$,  with support in
$B(\ve)$ and with $|D^\alpha L^*(u)(x)|\le C \|x\|^{2(k-1)-|\alpha|}$ for 
$|\alpha|\le k$. Thus we can proceed by induction to conclude that
\begin{equation}\label{estim6}
|I(\lambda)|\le C(f)\lambda^{-k}\sum_{|\alpha|\le k}\sup_{x\in\R^m}
\left|D^\alpha( g(x)\log\|x\|)\cdot\|x\|^{-2k+|\alpha|}\right|.
\end{equation}
We introduce the following auxiliary functions 
\begin{equation}\label{auxili}
f_s(x):=\|x\|^2+sR(x),\quad s\in[0,1]. 
\end{equation}
Then $f_1=f$. Let
\[
I(\lambda,s):=\int_{\R^m} e^{-\lambda f_s(x)}g(x)\log\|x\|\,dx.
\]
Differentiating in $s$ 2k times yields
\begin{equation}\label{deriv1}
I^{(2k)}(\lambda,s)=\lambda^{2k}\int_{\R^m}e^{-\lambda f_s(x)} R(x)^{2k}g(x)\log\|x\|
\;dx.
\end{equation}
Now let $g\in C^\infty(\R^m)$ with $\supp g\subset B(\ve)$. 
Let $u(x)=R(x)^{2k}g(x)\log\|x\|$. Then $u\in C^\infty(\R^m\setminus\{0\})$
with support in $B(\ve)$ and by \eqref{bound} it follows that 
$|D^\alpha u(x)|\le C \|x\|^{6k-|\alpha|}$, $\|x\|<\ve$. Applying \eqref{estim6}
with $3k$ in place of $k$, and \eqref{deriv1}, we get
\[
|I^{2k}(\lambda,s)|\le C\lambda^{-k}.
\]
By Taylor's theorem, we have
\begin{equation}
\left|I(\lambda,1)-\sum_{j<2k}\frac{1}{j!}I^{(j)}(\lambda,0)\right|\le
\sup_{s\in[0,1]}\frac{1}{(2k)!}\left| I^{2k}(\lambda,s)\right|.
\end{equation}
Thus modulo $\lambda^{-k}$ decay, it suffices to study the asymptotic expansion
of $I^{(j)}(\lambda,0)$. Each of these integrals is of the form
\begin{equation}\label{quadratic}
I_0(\lambda)=\int_{\R^m}e^{-\lambda\|x\|^2}g(x)\log\|x\|\;dx,
\end{equation}
where $g\in C^\infty_c(\R^m)$ with support in $B(\ve)$. Let $N\in\N$. By
Taylor's theorem we have
\[
\left|g(x)-\sum_{|\alpha|\le N}\frac{(D^\alpha g)(0)}{\alpha!}x^\alpha\right|\le
\left(\sum_{|\alpha|=N+1}\frac{1}{\alpha!}\sup_{x\in\R^m}|(D^\alpha g)(x)|\right)
\cdot\|x\|^{N+1}
\]
for $x\in B(\ve)$. Now for $\lambda\ge 1$ we have
\[
\begin{split}
\left|\int_{B(\ve)} e^{-\lambda\|x\|^2}\|x\|^{N+1}\log\|x\|\;dx\right|&\le 
\int_{\R^m}e^{-\lambda\|x\|^2}\|x\|^{N+1}\left|\log\|x\|\right|\;dx\\
&=O(\lambda^{-(m+N+1)/2}(1+\log\lambda)).
\end{split}
\]
Furthermore, for $\lambda\ge 1$ we have
\[
\int_{B(\ve)}e^{-\lambda\|x\|^2}x^\alpha\log\|x\|\;dx=\int_{\R^m}
e^{-\lambda\|x\|^2}x^\alpha\log\|x\|\;dx + O(e^{-\ve\lambda/2}).
\]
Changing variables $x\mapsto x/\sqrt{\lambda}$ in the integral on the right
hand side, we get
\[
\begin{split}
\int_{\R^m}e^{-\lambda\|x\|^2}x^\alpha\log\|x\|\;dx&=c_1\lambda^{-m/2-|\alpha|/2}
+c_2\lambda^{-m/2-|\alpha|/2}\log\lambda
\end{split}
\]
for some constants $c_1$ and $c_2$. Summarizing, it follows that for every
$N\in\N$ we have an expansion 
\begin{equation}\label{asympexp3}
I(\lambda)=\sum_{j=0}^N a_j\lambda^{-m/2-j/2}\log\lambda +\sum_{j=0}^N b_j
\lambda^{-m/2-j/2}+O(\lambda^{-(m+N)/2})
\end{equation}
as $\lambda\to\infty$. Now the integrals on the right hand side of 
\eqref{asympexp2} are of the form $I(1/(4t))$. Combining \eqref{weightint},
\eqref{estima5}, \eqref{asympexp2}, and  \eqref{asympexp3}, and using that
$m=d-1$, we get

\begin{prop}\label{asympexp4}
For every $N\in\N$ we have
\[
T_P^\prime(h^\nu_t)=\sum_{j=0}^Nc_j(\nu)t^{(j-1)/2}\log t+\sum_{j=0}^N
d_j(\nu)t^{(j-1)/2}+O(t^{(N+1)/2})
\]
as $t\to 0$. Moreover $c_1(\nu)=0$. 
\end{prop}
\begin{proof}
Only the last statement needs to be proved. The only terms that can make a 
contribution to $c_1(\nu)$ are the Taylor coefficients of $a_0^\nu(n(x))$ 
of degree 1, i.e.,
\[
c_1(\nu)=\sum_{i=1}^{d-1}\frac{\partial}{\partial x_i}a_0^\nu(n(x))\big|_{x=0}
\int_{\R^{d-1}} e^{-\|x\|^2}x_i\;dx.
\]
By \eqref{leadcoeff} we have 
\[
a_0^\nu(n(x))=\tr(k(n(x)))\cdot j(x_0,n(x)x_0)^{-1/2}.
\]
We have
\[
j(x_0,n(x)x_0)=\frac{\sinh(r(x))}{r(x)},
\]
where $r(x)$ is given by \eqref{dist3}. Using \eqref{distsq}, if follows that
\begin{equation}
\frac{\partial}{\partial x_i} j(x_0,n(x)x_0)^{-1/2}\big|_{x=0}=0,\quad i=1,\dots,d.
\end{equation}
\end{proof}
Combining Proposition \ref{asympexp4} with \eqref{Idcontr1}, \eqref{hyper-asymp}
and \eqref{parab-asymp}, Theorem \ref{theo-main} follows.

\section{Analytic torsion}\label{sec-analytor}
\setcounter{equation}{0}

Let $\tau$ be an irreducible finite dimensional representation of $G$ on
$V_{\tau}$. Let $E'_{\tau}$ be the flat
vector bundle associated to the
restriction of $\tau$ to $\Gamma$. Then $E'_{\tau}$ is canonically isomorphic to
the locally homogeneous vector bundle $E_{\tau}$ associated to $\tau|_{K}$. By
\cite{MM}, there exists an inner product $\left<\cdot,\cdot\right>$ on
$V_{\tau}$ such that
\begin{enumerate}
\item $\left<\tau(Y)u,v\right>=-\left<u,\tau(Y)v\right>$ for all
$Y\in\mathfrak{k}$, $u,v\in V_{\tau}$
\item $\left<\tau(Y)u,v\right>=\left<u,\tau(Y)v\right>$ for all
$Y\in\mathfrak{p}$, $u,v\in V_{\tau}$.
\end{enumerate}
Such an inner product is called admissible. It is unique up to scaling. Fix an
admissible inner product. Since $\tau|_{K}$ is unitary with respect to this
inner product, it induces a metric on $E_{\tau}$ which will be called admissible
too. Note that for each $p$, the vector bundle $\Lambda^{p}(E_{\tau})
=\Lambda^pT^*X\otimes E_\tau$ is associated with the representation
\begin{equation}\label{nutau}
\nu_{p}(\tau):=\Lambda^{p}\Ad^{*}\otimes\tau:\:K\rightarrow\GL(\Lambda^{p}
\mathfrak{p}^{*}\otimes V_{\tau}),
\end{equation}
i.e., there is a canonical isomorphism
\begin{align}\label{p Formen als homogenes Buendel}
\Lambda^{p}(E_{\tau})\cong\Gamma\backslash(G\times_{\nu_{p}(\tau)}(\Lambda^{p}
\mathfrak{p}^{*}\otimes V_{\tau})).
\end{align}
By \eqref{isomor1} there is an isomorphism
\begin{align}\label{isomor4}
\Lambda^{p}(X,E_{\tau})\cong C^{\infty}(\Gamma\backslash G,\nu_{p}(\tau)),
\end{align}
and a corresponding isomorphism for the $L^{2}$-spaces. Let
$\Delta_{p}(\tau)$ be the Hodge-Laplacian on $\Lambda^{p}(X,E_{\tau})$ with
respect to the admissible inner product. By (6.9) in \cite{MM} it follows that
with respect to \eqref{isomor4} one has
\begin{equation}\label{kuga}
\Delta_{p}(\tau)=-R_\Gamma(\Omega)+\tau(\Omega)\Id,
\end{equation}
where $R_\Gamma$ is the right regular representation. Let $A_{\nu_p(\tau)}$
be the differential operator induced by $-R_\Gamma(\Omega)$ in $C^\infty(X,E_\nu)$.
We note that $\tau(\Omega)$ is a scalar which can be computed as follows.
If $\Lambda(\tau)=k_1(\tau)e_1+\dots k_{n+1}(\tau)e_{n+1}$ is the highest weight
of $\tau$, then
\begin{equation}\label{tauomega}
\tau(\Omega)=\sum_{j=1}^{n+1}\left(k_j(\tau)+\rho_j\right)^2-\sum_{j=1}^{n+1}
\rho_j^2.
\end{equation} 
For $G=\Spin(2n+1,1)$ this was proved in \cite[sect. 2]{MP1}. For
$G=\SO_0(2n+1,1)$, the proof is similar. Thus we get
\begin{equation}\label{regtr5}
\Tr_{\reg}\left(e^{-t\Delta_p(\tau)}\right)=e^{-t\tau(\Omega)}
\Tr_{\reg}\left(e^{-tA_{\nu_p(\tau)}}\right).
\end{equation}
In order to define the analytic torsion, we have to determine the asymptotic
behavior of $\Tr_{\reg}\left(e^{-t\Delta_p(\tau)}\right)$ as $t\to 0$ and 
$t\to \infty$. By Theorem \ref{theo-main} it follows that 
there is an asymptotic expansion
\begin{equation}\label{asympexp7}
\Tr_{\reg}\left(e^{-t\Delta_p(\tau)}\right)\sim
\sum_{j=0}^\infty a_j(p,\tau)t^{-d/2+j/2}\log t
+\sum_{k=0}^\infty b_k(p,\tau)t^{-d/2+j/2}
\end{equation}
as $t\to 0$. Moreover $a_n(p,\tau)=0$. To determine the asymptotic behavior of
the regularizes trace as $t\to \infty$, we use \eqref{regtrace2}.
Let $0\le\lambda_1\le\lambda_2\le
\cdots$ be the eigenvalues of $\Delta_p(\tau)$. By \eqref{kuga} and
\eqref{regtrace2} we have
\begin{equation}\label{regtrace5}
\begin{split}
\Tr_{\reg}\left(e^{-t\Delta_p(\tau)}\right)&=\sum_{j}e^{-t\lambda_j}+
\sum_{\substack{\sigma\in\hat{M};\sigma=w_0\sigma\\
\left[\nu_p(\tau):\sigma\right]\neq 0}}
e^{-t(\tau(\Omega)-c(\sigma))}\frac{\Tr(\widetilde{\boldsymbol{C}}(\sigma,
\nu_p(\tau),0))}{4
}\\
&-\frac{1}{4\pi}\sum_{\substack{\sigma\in\hat{M}\\
\left[\nu_p(\tau):\sigma\right]\neq
0}}e^{-t(\tau(\Omega)-c(\sigma))}\\
&\hskip20pt\cdot\int_{\R}e^{-t\lambda^2}
\Tr\left(\widetilde{\boldsymbol{C}}(\sigma,\nu_p(\tau),-i\lambda)\frac{d}{dz}
\widetilde{\boldsymbol{C}}(\sigma,\nu_p(\tau),i\lambda)\right)\,d\lambda.
\end{split}
\end{equation}
Assume that $d=2n+1$. Let $h_p(\tau):=\dim(\ker\Delta_p(\tau)\cap L^2)$. 
Using \eqref{regtrace5} and \cite[Lemmas 7.1, 7.2]{MP2},  
it follows that there is an asymptotic expansion
\begin{equation}\label{asympexp5}
\Tr_{\reg}\left(e^{-t\Delta_p(\tau)}\right)\sim h_p(\tau)+\sum_{j=1}^\infty c_j 
t^{-j/2},\quad t\to\infty
\end{equation}

Combined with \eqref{asympexp7} we can define the spectral zeta function by
\begin{equation}\label{speczeta1}
\begin{split}
\zeta_p(s;\tau)&:=\frac{1}{\Gamma(s)}\int_0^1 t^{s-1}\left(\Tr_{\reg}\left(
e^{-t\Delta_p(\tau)}\right)-h_p(\tau)\right)\,dt\\
&\quad+\frac{1}{\Gamma(s)}\int_1^\infty t^{s-1}\left(\Tr_{\reg}\left(
e^{-t\Delta_p(\tau)}\right)-h_p(\tau)\right)\,dt.
\end{split}
\end{equation}
By \eqref{asympexp7} the first integral on the right converges in
the half-plane $\Re(s)>d/2$ and admits a meromorphic extension to $\C$ which
is holomorphic at $s=0$. By \eqref{asympexp5}, the second integral converges
in the half-plane $\Re(s)<1/2$ and also admits a meromorphic extension to $\C$
which is holomorphic at $s=0$.

Now assume that $\tau\not\cong\tau_{\theta}$. Let the highest weight 
$\Lambda(\tau)$ be given by \eqref{Darstellungen von G}. The highest weight
$\Lambda(\tau_\theta)$ of $\tau_\theta$ is given by
\[
\Lambda(\tau_\theta)=k_1(\tau)e_1+\cdots+k_n(\tau)e_n-k_{n+1}(\tau)e_{n+1}.
\]
Therefore, the condition $\tau\not\cong\tau_{\theta}$ implies 
$k_{n+1}(\tau)\neq 0$. 
Then by \cite[Lemma 7.1]{MP2} we have
$\tau(\Omega)-c(\sigma)>0$ for all $\sigma\in\hat M$ with 
$[\nu_p(\tau)\colon\sigma]\neq0$ and $p=0,\dots,d$.  
Furthermore by \cite[Lemma 7.3]{MP2} we have
$\ker(\Delta_p(\tau)\cap L^2)=0$, $p=0,\dots,d$.
By \eqref{regtrace5} it follows that there exist $C,c>0$ such
that for all $p=0,\dots,d$:
\begin{equation}
\Tr_{\reg}\left(e^{-t\Delta_p(\tau)}\right)\le C e^{-ct},\quad t\ge 1.
\end{equation}
Using \eqref{asympexp4}, it follows that 
$\zeta_p(s;\tau)$ can be defined as in the compact case by
\begin{equation}\label{speczeta2}
\zeta_p(s;\tau):=\frac{1}{\Gamma(s)}\int_0^\infty t^{s-1}
\Tr_{\reg}\left(e^{-t\Delta_p(\tau)}\right)\;dt.
\end{equation}
The integral converges absolutely and uniformly on compact subsets of 
$\Re(s)>d/2$ and admits a meromorphic extension to $\C$ which is holomorphic 
at $s=0$. We define the regularized determinant of $\Delta_p(\tau)$ as in the 
compact case by
\begin{equation}\label{regdet}
\det \Delta_p(\tau):=\exp\left(-\frac{d}{ds}\zeta_p(s;\tau)\big|_{s=0}
\right).
\end{equation}
In analogy to the compact case we now define the analytic torsion 
$T_X(\tau)\in\R^+$ associated to the the flat bundle
$E_{\tau}$, equipped with the admissible metric, by
\begin{equation}\label{anator3}
T_{X}(\tau):=\prod_{p=1}^{d}\det\Delta_{p}(\tau)^{(-1)^{p+1}p/2}.
\end{equation}
 Let
\begin{equation}\label{ktau5}
K(t,\tau):=\sum_{p=1}^{d}(-1)^{p}p\Tr_{\reg}(e^{-t\Delta_{p}(\tau)}).
\end{equation}
If $\tau\not\cong\tau_\theta$, then $K(t,\tau)=O(e^{-ct})$ as $t\to\infty$, and
the analytic torsion is given by
\begin{align}\label{Def Tor}
\log{T_{X}(\tau)}=\frac{1}{2}\frac{d}{ds}\biggr|_{s=0}\left(\frac{1}{\Gamma(s)}
\int_{0}^{\infty}t^{s-1}K(t,\tau)\,dt\right),
\end{align}
where  the right hand side is defined near $s=0$ by analytic continuation.

\end{document}